\documentclass[12pt,leqno]{amsart}
\usepackage{amsmath,amssymb,txfonts}
\usepackage{amssymb}
\usepackage{amsxtra}
\usepackage{amsmath}
\usepackage{txfonts}

 \usepackage{color}

\usepackage[cp1252]{inputenc}

\usepackage{mathrsfs}
\usepackage{graphicx}

 \usepackage[active]{srcltx}

\allowdisplaybreaks

\def\rr{{\mathbb R}}
\def\rn{{{\rr}^n}}

\def\fz{\infty}

\def\dist{{\mathop\mathrm{\,dist\,}}}

\def\boz{{\Omega}}

\def\diam{{\mathop\mathrm{\,diam\,}}}

\def\r{\right}
\def\lf{\left}

\newtheorem{thm}{Theorem}[section]
\newtheorem{lem}{Lemma}[section]

\newtheorem{cor}{Corollary}[section]
\newtheorem{defn}{Definition}[section]

\numberwithin{equation}{section}

\begin{document}
\arraycolsep=1pt

\title[Local Sobolev-Poincar\'e imbedding domains ]{Local Sobolev-Poincar\'e imbedding domains}

\author{Tian Liang and Zheng Zhu}

\address{Tian Liang\\
School of Mathematics and Statistics\\
Huizhou University\\
Guangdong 516007\\
P. R. China}
\email{\tt liangtian@hzu.edu.cn}

\address{Zheng Zhu\\
School of Mathematical Science\\
Beihang University\\
        Changping District Shahe Higher Education Park South Third Street No. 9\\
        Beijing 102206\\
        P. R. China}
\email{\tt zhzhu@buaa.edu.cn}


\thanks{The authors  would like to thank Prof. Yuan Zhou for lots of valuable discussions which helps this paper essentially. The authors also thank Prof. Tapio Rajala and Prof. Qingshan Zhou for pointing out some typos in the first version of manuscript and making the manuscript more readable. The first author is supported by the NSFC grant (No.12201238) and GuangDong Basic and Applied Basic Research Foundation (Grant No. 2022A1515111056). The second author is supported by the NSFC grant (No. 12301111) and the starting grant from Beihang University (ZG216S2329).}

\maketitle

\begin{abstract}
Impressed by the remarkable results about the characterization of (global) Sobolev-Poincar\'e imbedding domains in \cite{bsk96} by Buckley and Koskela. In this article, we give characterizations to local Sobolev-Poincar\'e imbedding domains. The main result reads as below.
\begin{enumerate}
\item For $1\leq p\leq n$, a bounded domain with slice condition is a local Sobolev-Poincar\'e domain of order $p$ if and only if it is a uniform domain. Combine it with results in \cite{bsk96}, there exists a global Sobolev-Poincar\'e imbedding domain of order $p$ which is not a local Sobolev-Poincar\'e imbedding domain of order $p$.
\item For $n<p<\fz$, a bounded domain with slice condition is a local Sobolev-Poincar\'e imbedding domain of order $p$ if and only if it is an $\alpha$-cigar domain for $\alpha=(p-n)/(p-1)$. Combine it with results in \cite{bsk96}, a domain is a local Sobolev-Poincar\'e imbedding domain of order $p$ if and only if it is a (global) Sobolev-Poincar\'e imbedding domain of order $p$.
\end{enumerate}
\end{abstract}

\section{Introduction\label{s1}}
 In this paper, we give geometrical characterizations to bounded domains which support local Sobolev-Poincar\'e imbeddings. In this paper, if we do not explain deliberately, we always assume domains are bounded. First, we review some existing results about domains that support (global) Sobolev-Poincar\'e imbeddings.  If a bounded domain $\boz\subset\rn$ has some sufficiently nice geometrical property, for example the boundary of $\boz$ is smooth, then this domain supports (global) Sobolev-Ppoincar\'e imbeddings. The classical Sobolev-Poincar\'e imbeddings for sufficiently nice bounded domain depend on weather the exponent $p$ is less than or equaivalent to or larger than the dimension $n$. For $1\leq p<n$, for every function $u$ in the Sobolev class $W^{1, p}(\boz)$, we have the following Sobolev-Poincar\'e inequality
\begin{equation}{\label{Eq1}}
\inf\limits_{c\in\rr}\lf(\int_\boz\lf|u(x)-c\r|^{\frac{np}{n-p}}dx\r)^{\frac{n-p}{np}}\leq C\lf(\int_\boz|\nabla u(x)|^pdx\r)^{\frac{1}{p}},
\end{equation}
with a constant $C$ independent of $u$. By making use of the triangle inequality, the inequality is equivalent to the following inequality
\begin{equation}\label{eq:0}
\lf(\int_\boz|u(x)-u_\boz|^{\frac{np}{n-p}}dx\r)^{\frac{n-p}{np}}\leq C\lf(\int_\boz|\nabla u(x)|^pdx\r)^{\frac{1}{p}}
\end{equation}
 with maybe a different constant $C$ which is also independent of $u$, where $u_\boz$ means the integral average of $u$ over the domain $\boz$. This inequality was first proved by Sobolev in \cite{Sobolev1, Sobolev2} for $1<p<n$. For $p=1$, it was proved by Gagliardo \cite{G:1958} and Nirenberg \cite{N:1959} (also see \cite{M:1960} by Maz'ya).

 For $p=n$, for a sufficiently nice domain $\boz\subset\rn$ and every $u\in W^{1, n}(\boz)$, one gets the following Trudinger inequality
 \begin{equation}\label{eq:1}
 \|u-u_\boz\|_{\phi(L)(\boz)}\leq C\lf(\int_\boz|\nabla u(x)|^ndx\r)^{\frac{1}{n}}
 \end{equation}
with a constant $C$ independent of $u$. Please refer to \cite{T:1967} for a proof. Here
 \[\phi(t)=\exp\lf(t^{\frac{n}{n-1}}\r)-1\]
 and $\|\cdot\|_{\phi(L)(\boz)}$ is the corresponding Orlicz norm on $\boz$ defined by
 \[\|f\|_{\phi(L)(\boz)}:=\inf\lf\{s>0:\int_\boz\phi\lf(\frac{|f(x)|}{s}\r)dx\leq 1\r\}.\]
 It is well-known that the inequality (\ref{eq:1}) is equivalent to the following inequality
 \begin{equation}\label{Eq2}
 \inf\limits_{c\in\rr}\|u-c\|_{\phi(L)(\boz)}\leq C\lf(\int_\boz|\nabla u(x)|^ndx\r)^{\frac{1}{n}}
 \end{equation}
 with maybe a different constant $C$ independent of $u$, see \cite{M:book85}.

 For $n<p<\fz$, for certain domains $\boz$ (including all balls) and $u\in W^{1, p}(\boz)$, one gets the following H\"older imbedding inequality (the so-called Morry's inequality)
 \begin{equation}\label{eq:2}
 |u(x_1)-u(x_2)|\leq C|x_1-x_2|^{1-\frac{n}{p}}\lf(\int_\boz|\nabla(x)|^pdx\r)^{\frac{1}{p}}
 \end{equation}
 for almost every $x_1,x_2\in\boz$ and a constant $C$ independent of $u$. 

 As we know, Sobolev-Poincar\'e imbedding inequalities (\ref{eq:0}), (\ref{eq:1}) and (\ref{eq:2}) hold on sufficiently nice domains (like balls). We say a domain is a Sobolev-Poincar\'e imbedding domain of order $p$, if the corresponding Sobolev-Poincar\'e inequality holds on the domain. By making use of the result from \cite{HK:1991}, it is not difficult to show bounded $W^{1, p}$-extension domains for $1\leq p<\fz$ are Sobolev-Poincar\'e imbedding domains of order $p$. Roughly, a domain $\boz\subset\rn$ is called a $W^{1, p}$-extension domain, if every function in $W^{1, p}(\boz)$ can be extended to be a function in $W^{1, p}(\rn)$ whose norm is controlled from above uniformly by the norm of the original function. Please see the precise definition in next section. Then the following natural question arises.
 \begin{center}
 \textbf{ What is the geometric characterizations of Sobolev-Poincar\'e imbedding domains?}
  \end{center}
For $1\leq p<n$, Bojarski proved that John domains are Sobolev-Poincar\'e imbedding domains of order $p$, see \cite{b89}. Also see \cite{HK:1998} for this result. Roughly, a domain is called a John domain, if every point in the domain can be connected to a center point by a twisted cone whose size is comparable to the Euclidean distance from the point to the center point. Please see the precise definition below. Conversely, under some very weak connecting assumption that is the slice condition, Buckley and Koskela \cite{bk95} proved that a Sobolev-Poincar\'e imbedding domain of order $p$ for $1\leq p<n$ must be a John domain. Please see the prercise definitions of the slice property with respect to a point and the slice condition for a domain in next section. By the result in \cite{gm85b} from Gehring and Martio, if a domain is quasiconformally equivalent to the unit ball, it satisfies the slice condition. The Riemann Mapping Theorem implies that every proper simply connected domain in $\rr^2$ satisfies the slice condition. For $p=n$, Smith and Stegenga proved that $0$-carrot domains are Sobolev-Poincar\'e imbedding domains of order $n$. Conversely, Buckley and Koskela proved that if a domain satisfying the slice condition is a Sobolev-Poincar\'e imbedding domain of order $n$, then it must be a $0$-carrot domain. See the precise definition of $\alpha$-carrot domain for $0\leq\alpha<1$ below. By the definition, it is easy to verify that John domains are $0$-carrot domains. For $n<p<\fz$, under the assumption that the domain satisfies the slice condition, Buckley and Koskela \cite{bk95} proved that a domain is a Sobolev-Poincar\'e imbedding domain of order $p$ if and only if it is an $\alpha$-cigar domain for $\alpha=(p-n)/(p-1)$. See the precise definition of cigar domains below. In the paper \cite{bk95}, Buckley and Koskela always used the terminologies that weak carrot domains and weak cigar domains. To simplify the notation, we just say carrot domains and cigar domains here. In \cite{z11}, Zhou generalized these results to the fractional Haj\l{}asz-Sobolev spaces.

 A domain $\boz\subset\rn$ is called a local Sobolev-Poincar\'e imbedding domain of order $p$, if it satisfies a local version of one of the corresponding inequalities (\ref{Eq1}), (\ref{Eq2}) and (\ref{eq:2}). To be more precise, a domain $\boz\subset\rn$ is called a local Sobolev-Poincar\'e imbedding domain of order $p$ for $1\leq p<n$, if for every $ u\in{W}^{1,p}(\Omega)$ and every $x \in\Omega$ and $0<r <\diam(\boz)$, the inequality
\begin{eqnarray}\label{e1.w1}
\inf \limits _{c\in\mathbb{R}}\lf(\int_{B(x, r)\cap\boz}|u(y)-c|^{\frac{np}{n-p}}dy\r)^{\frac{n-p}{np}}\le C\lf(\int_{B(x, \lambda r)\cap\boz}|\nabla u(y)|^pdy\r)^{\frac{1}{p}}
\end{eqnarray}
holds with constants $C, \lambda\geq 1$ independent of $u$ and $B(x, r)$. For $p=n$, if for every $ u\in{W}^{1,n}(\Omega)$ and every $x \in\Omega$ and $0<r < \diam(\boz)$, the inequality
\begin{eqnarray}\label{e1.w2}
\inf \limits_{c\in\mathbb{R}} \int_{B(x,r)\cap\Omega} \exp\left(\frac{A|u(y)-c|}{\|\nabla u\|_{L^{n}(B(x, \lambda r)\cap\Omega)}}\right)^{\frac{n}{n-1}}\,dy
\le C|B(x, r)|
\end{eqnarray}
holds with constants $C, \lambda\geq 1$ and $A>0$ independent of $u$ and $B(x, r)$. For $n<p<\infty $, if for every $ u\in{W}^{1,p}(\Omega)$ and every $x \in\Omega$ and $0<r <\diam(\boz)$, the inequality
\begin{eqnarray}\label{e1.w3}
|u(x_1)- u(x_2)| \leq C|x_1-x_2|^{1-\frac{n}{p}}\lf(\int_{B(x, \lambda r)\cap\boz}|\nabla u(y)|^pdy\r)^{\frac{1}{p}}.
\end{eqnarray}
holds for almost every $x_1, x_2\in B(x, r)\cap\boz$ and constants $C, \lambda\geq 1$ independent of $u$ and $B(x,r)$. As we explained before, the geometric characterization of Sobolev-Poincar\'e imbedding domains is already clear. For local Sobolev-Poincar\'e imbedding domains, a similar natural question arises.
 \begin{center}
\textbf{ What is the geometric characterization of local Sobolev-Poincar\'e imbedding domains?}
 \end{center}
For bounded domains, by taking the radius $r$ to be large enough, one can see that a local Sobolev-Poincar\'e imbedding domain of order $p$ is also a (global) Sobolev-Poincar\'e imbedding domain of order $p$, for every $1\leq p<\fz$. But the converse is not correct for $1\leq p\leq n$. The planar slit disk is a typical countexample. It is the unit disk with one radial segment removed. The slit disk is a typical John domain with slice condition, so it is a Sobolev-Poincar\'e imbedding domain of order $p$ for every $1\leq p\leq n$. See the results in \cite{bsk96}. However, it is not difficult to check that the slit disk is not a local Sobolev-Poincar\'e imbedding domain of order $p$ for any $1\leq p<\fz$. Hence, for every $1\leq p\leq n$, the class of local Sobolev-Poincar\'e imbedding domains of order $p$ must be a proper subclass of (global) Sobolev-Poincar\'e imbedding domains of order $p$. The class of uniform domains is a subclass of John domains. Uniform domains were first introduced by Martio and Sarvas \cite{MS79}. Roughly, a domain is called a uniform domain, if arbitrary two points inside the domain can be connected by a twisted double cone whose size is comparable to the distance between these two points. Please see the precise definition below. In \cite{J81}, Jones proved that uniform domains are $W^{k, p}$-extension domains for arbitrary $k\in\mathbb N$ and $1\leq p\leq\fz$.  In the main result of the paper, we will show a bounded domain with slice condition is a local Sobolev-Poincar\'e imbedding domain of order $p$ for $1\leq p\leq n$ if and only if it is a uniform domain. 
\begin{thm}\label{t1.1}
 If $\boz\subset\rn$ is a bounded domain with the slice condition, then it is a local Sobolev-Poincar\'e imbedding domain of order $p$ with $1\leq p\leq n$ if and only if it is a uniform domain.
\end{thm}
A result in \cite{BD:2007} by Buckley and Herron tells us that uniform domains are $(LLC)$ and if a $(LLC)$ domain is also quasiconvex and satisfies the slice condition, then it is also uniform. Here $LLC$ is the abbreviation of ``locally linearly connection". Please see the precise definition in next section. 
As we discussed before, for $1\leq p\leq n$, the class of local Sobolev-Poincar\'e imbedding domains of order $p$ is a proper subclass of (global) Sobolev-Poincar\'e imbedding domains of order $p$. However, the next theorem says they are the same for $n<p<\fz$.
\begin{thm}\label{th:2}
Let $\boz\subset\rn$ be a bounded domain with the slice condition. Then $\boz$ is a local Sobolev-Poincar\'e imbedding domain of order $p$ for $n<p<\fz$ if and only if it is an $\alpha$-cigar domain with $\alpha=(p-n)/(p-1)$.
 \end{thm}
 
 As we discussed before, bounded planar simply connected domains satisfy the slice condition. Hence, if we only consider planar bounded simply connected domains, we can obtain the following corollary.
\begin{cor}
If $\boz\subset\rr^2$ is a bounded simply connected domain, then
\begin{itemize}
\item for $1\leq p\leq n$, it is a local Sobolev-Poincar\'e imbedding domain of order $p$ if and only if it is a uniform domain;

\item for $n<p<\fz$, it is a local Sobolev-Poincar\'e imbedding domain of order $p$ if and only if it is an $\alpha$-cigar domain with $\alpha=(p-n)/(p-1)$.
\end{itemize} 
\end{cor}

 The paper is arranged as following. In section $2$, we will give definitions, recall some known results and prove some lemmas which will be used later. In section 3, we will prove Theorem \ref{t1.1} and Theorem \ref{th:2}.

\section{Preliminary}
We write $C$ as generic positive constants. These constants may change even in a single string of estimates. The dependence of the constant on a parameters $\alpha, \beta,\cdots$ is expressed by the notation $C = C(\alpha, \beta, \cdots)$ if needed. For two points $x, y\in\rn$, $|x-y|$ means the Euclidean distance between them. For a Lebesgue measurable set $E\subset\rn$, $|E|$ means the Lebesgue measure of it. For a measurable function $u$ defined on $E$ with $0<|E|<\fz$, we define
\[u_E:=\fint_Eu(x)dx:=\frac{1}{|E|}\int_Eu(x)dx\]
to be the integral average of $u$ over the set $E$. A continuous mapping $\gamma:[a, b]\to\boz\subset\rn$ is called a curve. Recall that the length of a curve $\gamma:[a, b]\to\boz$ is the total variation over $[a, b]$
\[l(\gamma):=\sup\lf\{\sum_{i=0}^{n-1}\lf|\gamma(t_i)-\gamma(t_{i+1})\r|\r\}\]
where the supremum is taken over all finite partitions 
\[a=t_0<t_1<\cdots<t_n=b\]
of the interval $[a, b]$. If $l(\gamma)<\fz$, the curve $\gamma$ is called rectifiable. Recall that every rectifiable curve $\gamma$ admits a parametrization by  the arc-length that $\gamma_0:[0, l(\gamma)]\to\boz$. For the arc-length parametrization $\gamma_0$ of $\gamma$, we have 
\[l\lf(\gamma_o\big|_{[t_1, t_2]}\r)=t_2-t_1\]
for all $0\leq t_1\leq t_2\leq l(\gamma)$. The integral of a Borel function $\rho$ over a rectifiable curve $\gamma$ is usually defined via the arc-length parametrization $\gamma_0$ of $\gamma$ by setting
\[\int_\gamma\rho(z)|d(z)|:=\int_{[0, l(\gamma)]}\rho\circ\gamma_0(t)dt. \]
 From now on, we only consider curves that are arc-length parametrized.  For a curve $\gamma\subset\boz$ and $x, y\in\boz$, we use the notation $\gamma^{[x, y]}$ to denote the subcurve of $\gamma$ between $x$ and $y$. A domain $\boz\subset\rn$ is called quasiconvex, if there exists a large enough constant $C>1$ such that for every $x, y\in\boz$, there exists a rectifiable curve $\gamma\subset\boz$ joining $x$ and $y$ such that 
 \[l(\gamma)\leq C|x-y|.\]

Let us give some definitions next. First, let us recall the definition of John domains.
\begin{defn}\label{de:john}
A domain $\boz\subset\rn$ is called a $c_0$-John domain for some constant $0<c_0\leq 1$, if it is bounded and there exists a distinguished center point $x_0\in\boz$ such that for every $x\in\boz$, there exists a curve $\gamma:[0,l]\to\boz$ parametrized by arclength such that $\gamma(0)=x$, $\gamma(l)=x_0$ such that for every $t\in [0, l]$, we have
\[\dist(\gamma(t), \partial\boz)\geq c_0t.\]
\end{defn}
Roughly, every point inside a John domain can be connected to a center by a twisted cone whose size is comparable to the Euclidean distance between $x$ and $x_0$. Next we give the definition of uniform domains.
\begin{defn}\label{de:uniform}
A domain $\Omega \subset \rn $ is called $\epsilon_{0}$-uniform for $0<\epsilon_0\leq1$, if for all $x,y \in \Omega$,
there exists a rectifiable curve $\gamma\subset\Omega$ joining $x,y$ with
\begin{equation}
 l(\gamma) \leq \frac{1}{\epsilon_{0}}|x-y|,\nonumber
 \end{equation}
 and for every $z\in\gamma$, we have
 \begin{equation}
 \dist(z,\partial  \Omega) \geq \epsilon_{0}\frac{|x-z||y-z|}{|x-y|}.\nonumber
\end{equation}
\end{defn}
Roughly, every pair of points inside the uniform domain can be connected by a twisted double cone whose size is comparable to the Euclidean distance between $x$ and $y$. By the definition, it is easy to see that every bounded uniform domain is also a John domain. The planar slit disk is a typical John domain which is not a uniform domain. Indeed, under the assumption with slice condition, to prove a local Sobolev-Poincar\'e imbedding domain of order $p$ for $1\leq p\leq n$ is also a uniform domain, we will prove that it is quasiconvex and satisfies the following locally linear connection (for short $LLC$) condition. 
\begin{defn}\label{de:LLC}
We say a domain $\Omega \subset \rn$ satisfies locally linear connection (for short $LLC$) condition if there exists a constant $b \in (0,1]$ such that for every $ z \in \rn $ and $ r>0 $, we have
\begin{itemize}
\item $LLC(1)$: Points in $\Omega \cap B(z, r)$ can be joined by a locally rectifiable curve in $\Omega \cap B(z , \frac{r}{b})$;
\item $LLC(2)$: Points in $\Omega \setminus B(z, r)$ can be joined by a locally rectifiable curve in $\Omega \setminus B(z , br)$.
 \end{itemize}
\end{defn}
A homeomorphism $h:\boz\subset\rn\to\boz'\subset\rn$ in the class $W^{1, n}_{\rm loc}(\boz, \boz')$ is called a quasiconformal mapping, if there exists a constant $K>1$ such that for almost every $z\in\boz$, we have
\[|Dh(z)|^n\leq K|J_h(z)|,\]
where $|Dh(z)|$ means the operator norm of the differential matrix $Dh(z)$ and $|J_h(z)|$ is the absolutely value of the Jacobian determinate. A domain $\boz$ is called quasiconformally equivalent to another domain $\boz'$, if there exists a quasiconformal mapping from $\boz$ onto $\boz'$. According to the results in \cite{gm85b} by Gehring and Martio, a domain quasiconformally equivalent to a uniform domain is $LLC$; conversely, if a $LLC$ domain is quasiconvex and satisfies the following slice condition, it is also a uniform domain. Now, let us give the definition of the slice condition.
\begin{defn}\label{de:slice}
Let $\boz\subset\rn$ be a domain with a distinguished point $x_0$ and $C>1$ be a constant. We say that $\boz$ has the $C$-slice property with respect to $x_0$, if for every $x\in\boz$, there is a curve $\gamma:[0, 1]\to\boz$ with $\gamma(0)=x_0$  and $\gamma(1)=x$, and there is a pairwise disjoint collection of open subsets $\{U_i\subset\boz:i=0,\cdots, j\}$ such that:
\begin{enumerate}
\item $x_0\in U_0$, $x\in U_j$, and $x_0$ and $x$ are in different components of $\boz\setminus\overline{U_i}$ for all $0<i<j$;

\item If $\tilde\gamma\subset\subset\boz$ is a curve containing both $x$ and $x_0$, then for all $0<i<j$, we have
\[\diam(U_i)\leq Cl(\tilde\gamma\cap U_i);\]

\item For every $t\in [0, 1]$, we have
\[B\lf(\gamma(t), \frac{\dist(\gamma(t), \partial\boz)}{C}\r)\subset\bigcup_{i=0}^j\overline{U_i};\]

\item For every $0\leq i\leq j$, we have
\[\diam(U_i)\leq C\dist(\gamma(t), \partial\boz)\]
for all $\gamma(t)\in\gamma\cap U_i.$ Also there exists $x_i\in\gamma\cap U_i$ as $x_0$ is previously defined and $x_j=x$ such that
\[B\lf(x_i, \frac{\dist(x_i, \partial\boz)}{C}\r)\subset U_i.\]
\end{enumerate}
We say that $\boz$ satisfies the slice condition if it has the $C$-slice property with respect to every point $x\in\boz$ for some $C>1$.
\end{defn}
There are also other kinds of slice conditions, please see the article \cite{B:2003} for an introduction. The following lemma was proved by Buckley and Herron in \cite{BD:2007}.
\begin{lem}\label{le:unisli}
A domain $\boz\subset\rn$ is a uniform domain if and only if it is quasiconvex, $LLC$ and satisfies the slice condition.
\end{lem}

Let us give the definition of cigar domains and carrot domains. We follow the definition given by Buckley and Koskela in \cite{bsk96} where they called them weak cigar domains and weak carrot domains. To simplify the notation, we just say cigar domains and carrot domains. First, let us give the definition of cigar domains.
\begin{defn}\label{de:cigar}
Let $0<\alpha<1$ and $0<\beta\leq\alpha$. A domain $\boz\subset\rn$ is called an $(\alpha, \beta)$-cigar domain, if for every pair of points $x, y\in\boz$, there exists a curve $\gamma:[0, 1]\to\boz$ such that $\gamma(0)=x$, $\gamma(1)=y$ and the inequality
\[\int_{\gamma}\dist(z, \partial\boz)^{\alpha-1}|dz|\leq C|x-y|^\beta\]
holds with a constant $C$ independent of $x, y$. A domain $\boz\subset\rn$ is called a $(0, 0)$-cigar domain, if for every pair of points $x, y\in\boz$, there exists a curve $\gamma:[0, 1]\to\boz$ such that $\gamma(0)=x$, $\gamma(1)=y$ and the inequality
\[\int_\gamma\frac{1}{\dist(z, \partial\boz)}|dz|\leq C\log\lf(1+\frac{|x-y|}{\min\{\dist(x, \partial\boz), \dist(y, \partial\boz)\}}\r)\]
holds with a constant $C$ independent of $x, y$.
\end{defn}
To simplify the notation, we use $\alpha$-cigar domain for $(\alpha, \alpha)$-cigar domain. The class of $(0, 0)$-cigar domains coincides with the class of uniform domains, see \cite{GO79, J80}. It was shown in \cite{gm85a} that when $\alpha>0$, the class of $(\alpha, \beta)$-cigar domains is strictly larger than the class of uniform domains. In \cite{L85}, Lappalainen proved that every $\alpha$-H\"older continuous function defined on an $(\alpha, \beta)$-cigar domain can be extended to be a global $\beta$-H\"older continuous function. Now, let us give the definition of carrot domains.
\begin{defn}\label{de:carrot}
Let $0\leq \alpha<1$. A domain $\boz\subset\rn$ is called an $\alpha$-carrot domain, if there exists a distinguished center $x_0\in\boz$ such that for every $x\in\boz$, there exists a curve $\gamma:[0, 1]\to\boz$ such that $\gamma(0)=x$, $\gamma(1)=1$ and the inequality
\begin{equation}\label{eq:carrot}
\int_{\gamma}\dist(z, \partial\boz)^{\alpha-1}|dz|\leq
\begin{cases}
C, &\ \ {\rm if}\ 0<\alpha<1,\\
C\lf(1+\log\lf(\frac{C_1}{\dist(x, \partial\boz)}\r)\r), &\ \ {\rm if}\ \alpha=0,
\end{cases}
\end{equation}
holds for constants $C, C_1$ independent of $x$.
\end{defn}
By the definition, for arbitrary $0\leq\alpha<1$, an $\alpha$-carrot domain must be bounded. It is easy to verify that John domains are $\alpha$-carrot domains for all $0\leq\alpha<1$. The converse is not correct. It is not hard to construct an $\alpha$-carrot domain which is not a John domain, see \cite{gm85a, bsk96}.  The condition in the definition of carrot domains is widely known as quasihyperbolic boundary condition, see \cite{k982} and references therein.

The definitions of uniform domains and John domains show these two classes of domains are closely relative. The next lemma shows that a uniform domain is not far from John domain locally. The case $\diam(\boz)=\fz$ was already considered in \cite[Lemma 1.9]{t86}. It seems their argument is not easy to be extended to the bounded case. So we prove the result for the case $\diam(\boz)<\fz$.
\begin{lem}\label{l2.9}
Suppose that $\Omega$ is an $\epsilon_0$-uniform domain with  $\diam \Omega < \infty $.
Given any $x_0 \in \Omega$ and $r>0$,
there is a $c_0$-John domain $\Omega_r \subset \Omega$ and a constant $\lambda>1$  such that
$$B(x_0,r) \cap \Omega \subset \Omega_r \subset B(x_0, \lambda r) \cap \Omega.$$
Moreover, the constant $c_0$ only depends on the constant $\epsilon_0$.
\end{lem}

\begin{proof}
Assume there exists a point $y_0 \in \Omega$  such that $|x_0-y_0|\geq \frac{4r}{\epsilon_{0}}$. Otherwise, let $\lambda = \frac{4}{\epsilon_{0}}$, we have
\[B(x_0, \lambda r) \cap \Omega= \Omega.\]
The result in \cite[Lemma 2.18]{gm85a} says that a bounded uniform domain is always a John domain. Hence, we can just take $\Omega_r =\Omega$ to obtain
$$B(x_0,r) \cap \Omega \subset \Omega_r  \subset B(x_0, \lambda r) \cap \Omega.$$

In the other case, fix $y_0 \in \Omega$ with $|x_0-y_0|= \frac{4r}{\epsilon_{0}}$. There exists a curve $\gamma_0:[0, 1]\to\boz$ with $\gamma_0(0)=x_0$, $\gamma_0(1)=y_0$ and $\gamma_0$ satisfies the conditions in Definition \ref{de:uniform}.
Set
\begin{align*}
G_0:=\left\{z \in \Omega\, :\exists \, v \in \gamma_0 \,such\,\, that \, |z-v| < \frac{ \epsilon_0|x_0-v||y_0-v|}{|x_0-y_0|} \right\}.
\end{align*}
Then $G_0 \subset \Omega$ contains a point $z_0 \in \gamma_0$ with $|x_0-z_0| = \frac{2r}{\epsilon_{0}}$ and $B(z_0,r) \subset G_0 $.
Indeed, the triangle inequality gives
\begin{align*}
\epsilon_0 \frac{|x_0-z_0||y_0-z_0|}{|x_0-y_0|} \geq\epsilon_0 \frac{|x_0-z_0|(|x_0-y_0|-|x_0-z_0|)}{|x_0-y_0|} \geq r.
\end{align*}
Fix $x \in (B(x_0,r)\cap \Omega)\backslash B(z_0, r)$. There exists a curve $\gamma_x:[0, 1]\to\boz$ with $\gamma_x(0)=z_0$ and $\gamma_x(1)=x$ and the curve $\gamma_x$ satisfies the conditions in Definition \ref{de:uniform}. Similar to $G_0$, define a domain $G_x \subset \Omega$ by setting
\begin{align*}
G_x:=\left\{z \in \Omega\,:\,\exists \, v \in \gamma_x\, such\, that\, |v-z| < \frac{ \epsilon_0|x-v||z_0-v|}{|x-z_0|}  \right\}.
\end{align*}
Next we prove that
\begin{align*}
\Omega_r:= \bigcup _{x\in (B(x_0,r)\cap \Omega)\backslash B(z_0, r)} G_x \cup B(z_0, r)
\end{align*}
is a John domain. For every $x\in (B(x_0, r)\cap\boz)\setminus B(z_0, r)$, we have
$$r \leq |x-z_0| \leq |x-x_0|+|x_0-z_0|\leq r\lf(1+\frac{2}{\epsilon_{0}}\r).$$
Hence, for every $z \in G_x$, there exists $v \in \gamma_x$ such that
\begin{align*}
|x_0-z|\leq |x_0-x|+|x-v|+|v-z|.
\end{align*}
Obviously, $|x_0-x|<r$. Denote $\gamma_x^{[x,v]}$ to be the subcurve of $\gamma_x$ from $x$ to $v$ and $\gamma_x^{[v, z_0]}$ to be the subcurve of $\gamma_x$ from $v$ to $z_0$. Since $\gamma_x$ satisfies the conditions in Definition \ref{de:uniform}, we have
\[|x-v|\leq l\lf(\gamma_x^{[x, v]}\r)\leq l(\gamma_x)\leq\frac{1}{\epsilon_0}|x-z_0|.\]
Since
\[l\lf(\gamma_x^{[x, v]}\r)+l\lf(\gamma_x^{[v, z_0]}\r)=l(\gamma_x),\]
by the definition of the set $G_x$, we have
\[|v-z|<\frac{\epsilon_0|x-v||z_0-v|}{|x-z_0|}\leq\frac{\epsilon_0l\lf(\gamma_x^{[x, v]}\r)l\lf(\gamma_x^{z_0, v}\r)}{|x-z_0|}\leq\frac{\epsilon_0\lf(l(\gamma_x)\r)^2}{4|x-z_0|}\leq\frac{\epsilon_0}{4}|x-z_0|.\]
Combine these inequalities together, we obtain
\[|x_0-z| \leq r + \frac{1}{\epsilon_{0}}|x-z_0| + \frac{\epsilon_{0}}{4}|x-z_0| \leq r\left(1+\frac{(\epsilon_{0}+4)^{2}}{4\epsilon_{0}^{2}}\right).\]
Hence, take 
\[\lambda=1+\frac{(\epsilon_{0}+4)^{2}}{4\epsilon_{0}^{2}},\]
we have
$$\Omega_r \subset B\left(x_0,r\left(1+\frac{(\epsilon_{0}+4)^{2}}{4\epsilon_{0}^{2}}\right)\right)\cap \Omega.$$

Now, let us prove that $\boz_r$ is a John domain with the distinguished center $z_0$. If $z\in B(z_0, r)$, it is not difficult to check that the segment $[z, z_0]$ between $z$ and $z_0$ satisfies the condition in Definition \ref{de:john} with the constant $c=1$. Let us consider the case that $z \in G_x \backslash B(z_0,r)$ for some $x\in(B(x_0, r)\cap\boz)\setminus B(z_0, r)$. By the definition of $G_x$, there exists a point $v \in \gamma_{x} $ such that
$$z \in B\lf(v, \frac{\epsilon_0  |x-v||z_0-v|}{|x-z_0|} \r)\subset G_x.$$
So $[z, v] \subset G_x$ and
$$[z, v]\cup \gamma^{[v, z_0]}_{x}\subset\boz_r$$
is a rectifiable curve from $z$ to $z_0$. There exists a unique $t\in(0, 1)$ such that
\[|\gamma_x(t)-v|=\frac{\epsilon_0|x-v||z_0-v|}{2|x-z_0|}\]
and for every $\tilde t\in(0, t)$, we have
\[|\gamma_x(\tilde t)-v|>\frac{\epsilon_0|x-v||z_0-v|}{2|x-z_0|}.\]
Denote $w:=\gamma_x(t)$. Then
\[\tilde\gamma:=[z, v]\cup[v, w]\cup\gamma_x^{[w, z_0]}\subset\boz_r\]
is also a rectifiable curve from $z$ to $z_0$.

Now, let us check the curve $\tilde\gamma$ satisfies the conditions in Definition \ref{de:john} with a constant $c=\frac{6(\epsilon_0+2)}{\epsilon_0^{3}}$ depending only on $\epsilon_0$. For every $s\in [z,v] \subset \tilde\gamma$, since $G_x\cup B(z_0, r)\subset\boz_r$ and
$$z \in B\lf(v, \frac{\epsilon_0  |x-v||z_0-v|}{|x-z_0|}\r)\subset G_x,$$
we have
\begin{align*}
d(s, \partial\boz_r)&\geq d(s, \partial (G_x \cup B(z_0,r)))\geq d(s, \partial G_x)\\
&\geq\frac{\epsilon_0  |x-v||z_0-v|}{|x-z_0|}-|s-v|\\
&\geq\frac{\epsilon_0  |x-v||z_0-v|}{|x-z_0|}-\left( |z-s|-|v-z|\right)\\
&\geq |z-s|=l(\tilde\gamma^{[z, s]}) \geq  \frac{ \epsilon_0^{3} l(\tilde\gamma^{[z, s]})}{6(\epsilon_0+2)}.
\end{align*}
For every $s\in [v, w]$
we know
\begin{align*}
d(s, \partial\boz_r)&\geq d(s, \partial (G_x \cup B(z_0,r) ))\geq d\left(s, \partial B\lf(v, \frac{\epsilon_0  |x-v||z_0-v|}{|x-z_0|} \right)\right)\\
&\geq \frac{\epsilon_0  |x-v||z_0-v|}{|x-z_0|} -|v-s|
\geq \frac{\epsilon_0  |x-v||z_0-v|}{2|x-z_0|}\\
 &\geq \frac{1}{3}\lf(|z-v|+|v-s|\r)
> \frac{\epsilon_0^{3}l(\tilde\gamma^{[z, s]})}{ 6(\epsilon_0+2)}.
\end{align*}
For every $s\in\tilde\gamma^{[w, z_0]}=  \gamma^{[w, z_0]}_{x} $, we have
$$|s-v|\geq \frac{\epsilon_0  |x-v||z_0-v|}{2|x-z_0|},$$
we have
$$|v-z| \leq 2|s-v| \leq 2l(\tilde\gamma^{[v,s]}).$$
We have
$$l(\tilde\gamma^{[z, s]})\leq l([z,v])+l(\gamma_{x}^{[x, s]})-l\gamma_{x}^{[x, v]} \leq 3l(\gamma_{x}^{[x, s]}).$$
If $|s-z_0|\geq \frac{r}{2}$, then we have
\begin{align*}
d(s, \partial\boz_r)&\geq d(s, \partial (G_x \cup B(z_0,r) )) \geq d(s, \partial G_x)\\
                             & \geq \frac{ \epsilon_0|x-s||z_0-s|}{|x-z_0|}
\geq \frac{ \epsilon_0^{2}|x-s|}{2(\epsilon_0+2)}
\geq \frac{ \epsilon_0^{3} l(\gamma_{x}^{[x, s]})}{6(\epsilon_0+2)}\geq \frac{\epsilon_0^3l(\tilde\gamma^{[z, s]})}{6(\epsilon_0+2)}.
\end{align*}
If $|s-z_0|\leq \frac{r}{2}$, since $B(z_0,r) \subset \Omega$, we have
\begin{align*}
d(s, \partial (G_x \cup B(z_0,r) )) \geq d(s, \partial B(z_0,r))
\geq r- |s-z_0| \geq \frac{r}{2}.
\end{align*}
Combine it with
$$|x-z_0| \leq r\lf(1+\frac{2}{\epsilon_{0}}\r),$$
we have
\begin{align*}
d(s, \partial\boz_r)&\geq d(s, \partial (G_x \cup B(z_0,r) )) \geq \frac{r}{2}\\
&\geq \frac{\epsilon_0|x-z_0|}{2(\epsilon_0+2)}
\geq \frac{\epsilon_0^{2}l(\gamma_{x})}{ 2(\epsilon_0+2)}
\geq \frac{\epsilon_0^{3}l(\tilde\gamma^{[z, s]})}{ 6(\epsilon_0+2)}.
\end{align*}
Combine all the estimates above, we can take $c_0={\epsilon_0^3}/{6(\epsilon_0+2)}$ and obtain the desired result.
\end{proof}

Let $\boz\subset\rn$ be a domain. We define the intrinsic metric $d_\boz$ on $\boz$ by setting
\[d_\boz(x_1,x_2):=\inf\lf\{l(\gamma):\gamma\subset\boz\ {\rm is\ a\ locally\ rectifiable\ curve\ joining}\ x_1\ {\rm and}\ x_2\r\}.\]
If there is no such curve, we simply set 
$$d_\boz(x_1, x_2)=\fz.$$ 
For arbitrary $x\in\boz$ and $r>0$, we define the intrinsic ball $B_\boz(x, r)$ by setting
\[B_\boz(x, r):=\lf\{y\in\boz: d_\boz(x, y)<r\r\}\]
and define the intrinsic boundary $\partial_\boz B_\boz(x, r)$ of the intrinsic ball $B_\boz(x, r)$ by setting
\[\partial_\boz B_\boz(x, r):=\lf\{y\in\boz: d_\boz(x, y)=r\r\}.\]

Obviously, $B_\boz(x, r)$ is a subset of $B(x, r)\cap\boz$. A domain $\boz\subset\rn$ is called Ahlfors $n$-regular, if there exists a constant $0<c<1$ and some $r_o>0$ such that for every $x\in\overline\boz$ and every $0<r<r_o$, we have
\[|B(x, r)\cap\boz|\geq c|B(x, r)|.\]
Similarly, we say that $\boz$ is intrinsic Ahlfors $n$-regular, if there exists a constant $0<c<1$ and some $r_o>0$ such that for every $x\in\overline\boz$ and $0<r<r_o$, we have
\[|B_\boz(x, r)|\geq c|B(x, r)|.\]
It is easy to see that intrinsic Ahlfors $n$-regularity can imply Ahlfors $n$-regularity, since $B_\boz(x, r)$ is a subset of $B(x, r)\cap\boz$.

By the result in \cite{hkt08} due to Haj\l{}asz, Koskela and Tuominen, we can conclude that a local Sobolev-Poincar\'e imbedding domain of order $p$ for $1\leq p<\fz$ is Ahlfors $n$-regular. This observation may not come from their result in the paper but comes from the proof in the paper directly. Nevertheless, we will prove a stronger result below. Of course, the idea behind the argument comes from the paper \cite{hkt08} by Haj\l{}asz, Koskela and Tuominen.

\begin{lem}\label{l2.3}
Let $\boz\subset\rn$ be a local Sobolev-Poincar\'e imbedding domain of order $p$ for $1\leq p<\fz$. Then $\boz$ is intrinsic Ahlfors $n$-regular.
\end{lem}

\begin{proof}
Fix $x\in\boz$ and $0<r<r_o$ with
\[r_o=\min\lf\{1, \lf(\frac{|\boz|}{2\omega_n}\r)^{\frac{1}{n}}\r\},\]
where $\omega_n$ means the volume of $n$-dimensional unit ball. Let $b_0 =1$ and $b_j \in (0,1) $ for $j \in N$ such that
\begin{equation}\label{equation1}
|B_\boz(x, b_j r)|=2^{-1}|B_\boz(x, b_{j-1 }r)|= 2^{-j} |B_\boz(x, r)|.
\end{equation}
For each $j\ge 0$, define a function $u_{x,b_{j+1}, b_j} $ on $\Omega$  by setting
\begin{align*}
u_{x,b_{j+1}, b_j}(y):=\left\{
\begin{array}{ll}
 1,        &    y \in B_\boz(x, b_{j+1} r), \\
\frac{b_{j}r-d_\boz(y, x)}{b_{j}r-b_{j+1} r},  \,\,   &   y \in B_\boz(x, b_{j}r)\setminus B_\boz(x, b_{j+1}r),\\
 0 ,      &      elsewhere.
\end{array} \right.
\end{align*}
Some simple computation gives
\begin{equation}
|\nabla u_{x, b_{j+1}, b_j}(y)|\leq\begin{cases}
\frac{1}{b_jr-b_{j+1}r},\  & \ y\in B_\boz(x, b_jr)\setminus B_\boz(x, b_{j+1}r),\\
0,\  &\ {\rm elsewhere}.
\end{cases}
\end{equation}
Some simple computation gives 
\begin{equation}
\int_\boz|u_{x, b_{j+1}, b_j}(y)|^pdy\leq|B_{\boz}(x, b_jr)|.\nonumber
\end{equation}
For every $\lambda \geq 1$, we have
\begin{eqnarray}\label{e2.1}
\left( \int_{\Omega} | \nabla u_{x,b_{j+1}, b_j}(y)|^{p} \, dy \right)^{\frac{1}{p}}=\left( \int_{ B(x,\lambda r) \cap \Omega} | \nabla u_{x,b_{j+1}, b_j}(y)|^{p} \, dy \right)^{\frac{1}{p}}\leq \frac{|B_\boz(x, b_{j}r)|^{\frac{1}{p}}}{b_{j}r- b_{j+1}r}.
\end{eqnarray}
Hence, $u_{x, b_{j+1}, b_j}\in W^{1, p}(\boz)$. According to the relationship between $p$ and $n$, we divide the argument below into three cases.

\noindent\underline{ \it Case ($1\leq p< n$): }
Let $j\geq 0$. For each $c \in \mathbb{R}$, by the definition of $u_{x, b_{j+1}, b_j}$, we have
$$u_{x,b_{j+1}, b_j}-c \geq \frac{1}{2}$$ either on $B_\boz(x,b_{j+1}r)$ or on $\Omega \setminus B_\boz(x,b_{j}r)$. Since $0<r<r_o$, we have
\[|\boz\setminus B_\boz(x, b_jr)|\geq |B_\boz(x, b_{j+1}r)|.\]
 Hence, we have
 \begin{equation}\label{eq:lower}
\inf \limits _{c\in\mathbb R} \left( \int_{B(x, r) \cap \Omega_0} |u_{x,b_{j+1}, b_j}(y)- c |^{\frac{np}{n-p}}\, dy \right)^{\frac{n-p}{np}} \geq\frac{1}{2}|B_\boz(x,b_{j+1}r)|^{\frac{n-p}{np}}.
\end{equation}
By combining (\ref{e1.w1}),  (\ref{equation1}), (\ref{e2.1}) and (\ref{eq:lower}), we obtain
\begin{eqnarray}
\left|B_\boz(x,b_{j}r)\right|^{\frac{n-p}{np}} \leq \frac{C|B_\boz(x, b_{j}r)|^{\frac{1}{p}}}{ b_{j}r- b_{j+1}r},
\end{eqnarray}
that is,
\begin{align*}
 b_{j}r- b_{j+1}r \leq C |B_\boz(x, b_{j}r)|^{\frac{1}{n}} \leq C 2^{-j/n}|B_\boz(x,r)|^{\frac{1}{n}}.
\end{align*}
Since $b_j \rightarrow 0$ as $ j \rightarrow \infty$, we have
\begin{align*}
r= \sum \limits _{j\geq 0} (b_{j}r- b_{j+1}r)
\leq C \sum \limits _{j \geq 0} 2^{-j/n}|B_\boz(x, r)|^{\frac{1}{n}}
\leq C |B_\boz(x, r)|^{\frac{1}{n}}.
\end{align*}
It implies $\boz$ is intrinsic Ahlfors $n$-regular.

\noindent\underline{ \it Case ($p= n$): }
As same as in the last case, for each $c \in\rr$, we have
$$| u_{x,b_{j+1}, b_j}-c| \geq \frac{1}{2}$$
either on $B_\boz(x,b_{j+1}r)$ or on $\boz \setminus B_\boz(x,b_{j}r)$.
Since
\[|\boz\setminus B_\boz(x, b_jr)|\geq |B_\boz(x, b_{j+1}r)|,\]
combine \eqref{e2.1} and \eqref{e1.w2}, we obtain
$$
|B_\boz(x,b_{j}r)| \exp \left(\frac{A( b_{j}r- b_{j+1}r)}{2|B_\boz( x, b_{j}r)|^{\frac{1}{n}}}\right)^{\frac{n}{n-1}} \leq C|B(x, r)|.
$$
It implies that
$$
b_{j}- b_{j+1}\leqslant \frac{2|B_\boz( x, b_{j}r)|^{\frac{1}{n}} }{Ar} \log^{\frac{n-1}{n}} \lf(\frac{C|B(x, r)|}{|B_\boz( x, b_{j}r)|}\r).
$$
This leads to
\begin{eqnarray*}
b_0&=& \sum \limits_ {j \geqslant 0}(b_{j}- b_{j+1})\\
& \leq& C \sum \limits_ {j \geqslant 0}\frac{|B_\boz( x, b_{j}r)|^{\frac{1}{n}}}{|B(x, r)|^{\frac{1}{n}}} \log^{\frac{n-1}{n}} \lf(\frac{|B(x, r)|}{|B_\boz( x, b_jr)|}\r)\\
&\leq& C \sum \limits_ {j \geqslant 0} 2^{-j/n}\frac{|B_\boz(x, r)|^{\frac{1}{n}}}{|B(x, r)|^{\frac{1}{n}}}\log^{\frac{n-1}{n}}\lf( \frac{2^{nj}|B(x, r)|}{|B_\boz( x, r)|}\r)\\
& \leq& C\sum\limits_{j\geq 0}2^{-j/n}\frac{|B_\boz(x, r)|^{\frac{1}{n}}}{|B(x, r)|^{\frac{1}{n}}}\lf(\log(2^{jn})+\log\lf(\frac{|B(x, r)|}{|B_\boz(x, r)|}\r)\r)^{\frac{n-1}{n}}\\
&\leq& C\sum\limits_{j\geq 0}2^{-j/n}\log^{\frac{n-1}{n}}(2^{jn})\frac{|B_\boz(x, r)|^{\frac{1}{n}}}{|B(x,r)|^{\frac{1}{n}}}\\
& & +C\sum\limits_{j\geq 0}2^{-j/n}\frac{|B_\boz(x, r)|^{\frac{1}{n}}}{|B(x, r)|^{\frac{1}{n}}}\log^{\frac{n-1}{n}}\lf(\frac{|B(x, r)|^{\frac{1}{n}}}{|B_{\boz}(x, r)|^{\frac{1}{n}}}\r)\\
&\leq& C\frac{|B_\boz(x, r)|^{\frac{1}{n}}}{|B(x, r)|^{\frac{1}{n}}}+C\frac{|B_\boz(x, r)|^{\frac{1}{n}}}{|B(x ,r)|^{\frac{1}{n}}}\log^{\frac{n-1}{n}}\lf(\frac{|B(x, r)|^{\frac{1}{n}}}{|B_\boz(x, r)|^{\frac{1}{n}}}\r)=:A_1+A_2.
\end{eqnarray*}
Hence, at least one of $A_1$ and $A_2$ must be not less than $\frac{1}{2}$.  If $A_1\geq\frac{1}{2}$, then we have
\[|B_\boz(x, r)|\geq cr^n\]
for some small constant $0<c<1$. It implies that $\boz$ is intrinsic Ahlfros $n$-regular. If $A_2\geq\frac{1}{2}$, then we have
\[\frac{|B_\boz(x, r)|^{\frac{1}{n}}}{|B(x, r)|^{\frac{1}{n}}}\log^{\frac{n-1}{n}}\lf(\frac{|B(x, r)|^{\frac{1}{n}}}{|B_\boz(x, r)|^{\frac{1}{n}}}\r)\geq\tilde c\]
for a sufficiently small $0<\tilde c<1$. Since
\[|B_\boz(x, r)|\leq|B(x, r)|,\]
it implies that there exists a small enough constant $c>0$ such that
\[|B_\boz(x, r)|\geq cr^n.\]
Hence, $\boz$ is intrinsic Ahlfors $n$-regular.

\noindent\underline{ \it Case ($n<p<\fz$): }
For every $j\geq 0$, fix a point $x_j\in\partial_\boz B_\boz(x, b_jr)$. By the definition of intrinsic metric, for arbitrary $\epsilon>0$, there exists a curve $\gamma_\epsilon\subset\boz$ joining $x$ and $x_j$ such that
\[l(\gamma_\epsilon)<b_jr+\epsilon.\]
Choose one point $\tilde x_j\in\gamma_\epsilon$ such that
\[l\lf(\gamma_\epsilon^{[\tilde x_j, x_j]}\r)=b_jr-b_{j+1}r+\epsilon.\]
Then the triangle inequality gives
\[\tilde x_j\in B_\boz(x, b_{j+1}r)\ {\rm and}\ |x_j-\tilde x_j|\leq b_jr-b_{j+1}r+\epsilon.\]
Apply inequalities (\ref{e1.w3}) and (\ref{e2.1}) to the function $u_{x, b_{j+1}, b_j}$, we obtain
\[1\leq C\lf(b_jr-b_{j+1}r+\epsilon\r)^{1-\frac{n}{p}}\frac{|B_\boz(x, b_jr)|^{\frac{1}{p}}}{b_jr-b_{j+1}r}.\]
Since $\epsilon>0$ is arbitrary, we have
\[b_jr-b_{j+1}r\leq C|B_\boz(x, b_jr)|^{\frac{1}{n}}.\]
Sum up $j$, we obtain
\begin{eqnarray*}
b_0r=\sum\limits_{j\geq 0}(b_jr-b_{j+1}r)\leq C\sum\limits_{j\geq 0}|B_\boz(x, b_jr)|^{\frac{1}{n}}\leq C\sum\limits_{j\geq 0}2^{-j/n}|B_\boz(x, r)|^{\frac{1}{n}}.
\end{eqnarray*}
It implies that there exists a small enough constant $0<c<1$ such that
\[|B_\boz(x, r)|\geq cr^n.\]
So $\boz$ is intrinsic Ahlfors $n$-regular.
\end{proof}

Some of recent research in analysis on metric measure space $(X, d, \mu)$ has focused on the geometric and analytic properties of Poincar\'e inequalities, for example, see \cite{DJS:2012, DJS:2016, DSW:2012, HK:1998, K:2002, KZ:2008} and other relative articles. Following \cite{HK:1998}, we say a metric measure space $(X, d, \mu)$ is a $p$-Poincar\'e space, if for every $u\in N^{1,p}(X)$, every upper gradient $\rho$ of $u$ and every ball $B\subset X$,  the $p$-Poincar\'e inequality 
\[\fint_{B}\lf|u-u_B\r|d\mu\leq C\diam(B)\lf(\fint_{\lambda B}\rho^pd\mu\r)^{\frac{1}{p}}\]
holds with constants $C, \lambda$ independent of $u$ and $B$. Here $\lambda B$ is the ball with the same center as $B$ and whose radius equals to $\lambda$ times the radius of $B$. Here $N^{1, p}(X)$ means the Newtonian space of order $p$ defined on $X$. The Newtonian space was first defined by Shanmugalingam in \cite{S:2000}. The Newtonian space generalizes the definition of Sobolev space to metric measure spaces.  Come back to an arbitrary Euclidean domain $\boz\subset\rn$, we always have that $N^{1, p}(\boz)$ equals to $W^{1, p}(\boz)$ for arbitrary $1\leq p<\fz$. Please refer to \cite{S:2000} for details. Implied by these remarkable works, we define the following Poincar\'e domains in the Euclidean space $\rn$.
\begin{defn}\label{de:poindom}
We say a domain $\boz\subset\rn$ is a $p$-Poincar\'e domain for $1\leq p<\fz$, if for every $u\in W^{1, p}(\boz)$ and $x\in\boz$ and $0<r<\diam(\boz)$, we have 
\begin{equation}\label{eq:poincare}
\fint_{B(x, r)\cap\boz}\lf|u(y)-u_{B(x, r)\cap\boz}\r|dy\leq Cr\lf(\fint_{B(x, \lambda r)\cap\boz}\lf|\nabla u(y)\r|^pdy\r)^{\frac{1}{p}}
\end{equation}
with constants $C, \lambda$ independent of $u$, $x$ and $r$.
\end{defn}
In next lemma, we will see that a local Sobolev-Poincar\'e imbedding domain of order $p$ is also a $p$-Poincar\'e domain.
\begin{lem}\label{le:SPtoP}
Let $\boz\subset\rn$ be a local Sobolev-Poincar\'e imbedding domain of order $p$ for $1\leq p<\fz$. Then it is also a $p$-Poincar\'e domain.
\end{lem}
\begin{proof}
According to $p$ is less than or equivalent to or larger than $n$, we divide the following argument into three cases.

\noindent\underline{ \it Case ($1\leq p< n$): } For every $u\in W^{1, p}(\boz)$, $x\in\boz$ and $0<r<\diam(\boz)$, we have 
\[\inf \limits _{c\in\mathbb{R}}\lf(\int_{B(x, r)\cap\boz}\lf|u(y)-c\r|^{\frac{np}{n-p}}dy\r)^{\frac{n-p}{np}}\le C\lf(\int_{B(x, \lambda r)\cap\boz}|\nabla u(y)|^pdy\r)^{\frac{1}{p}}\]
with constants $C, \lambda>1$ independent of $u, x$ and $r$. This inequality is equivalent to 
\[\lf(\int_{B(x, r)\cap\boz}\lf|u(y)-u_{B(x, r)\cap\boz}\r|^{\frac{np}{n-p}}dy\r)^{\frac{n-p}{np}}\leq C\lf(\int_{B(x, \lambda r)\cap\boz}|\nabla u(y)|^pdy\r)^{\frac{1}{p}}\]
with a constant $C$ which may be different to the one above. By Lemma \ref{l2.3}, $\boz$ is intrinsic Ahlfors $n$-regular, so the inequality above gives 
\[\lf(\fint_{B(x, r)\cap\boz}\lf|u(y)-u_{B(x, r)\cap\boz}\r|^{\frac{np}{n-p}}dy\r)^{\frac{n-p}{np}}\leq Cr\lf(\fint_{B(x, \lambda r)\cap\boz}|\nabla u(y)|^pdy\r)^{\frac{1}{p}}.\] 
Then the H\"older inequality implies
\[\fint_{B(x, r)\cap\boz}\lf|u(y)-u_{B(x, r)\cap\boz}\r|dy\leq Cr\lf(\fint_{B(x, \lambda r)\cap\boz}|\nabla u(y)|^pdy\r)^{\frac{1}{p}}.\]
It shows $\boz\subset\rn$ is a $p$-Poincar\'e domain.

\noindent\underline{\it Case ($p=n$):} For every $u\in W^{1, n}(\boz)$, $x\in\boz$ and $0<r<\diam(\boz)$, we have 
\[\inf \limits_{c\in\mathbb{R}} \int_{B(x,r)\cap\Omega} \exp\left(\frac{A|u(y)-c|}{\|\nabla u\|_{L^{n}(B(x, \lambda r)\cap\Omega)}}\right)^{\frac{n}{n-1}}\,dy
\le C|B(x, r)|.\]
Since $\boz$ is intrinsic Ahlfors $n$-regular, the inequality above is equivalent to 
\[\fint_{B(x, r)\cap\boz}\exp\lf(\frac{A|u(y)-u_{B(x, r)\cap\boz}|}{\|\nabla u\|_{L^n\lf(B(x, \lambda r)\cap\boz\r)}}\r)^{\frac{n}{n-1}}dy\leq C.\]
It is easy to check the function 
\[f(t)=\exp(t^{\frac{n}{n-1}})\]
is convex on $(0, \fz)$.  Then the Jensen inequality and the fact that $\boz$ is intrinsic Ahlfors $n$-regular imply the following inequality that 
\[\fint_{B(x, r)\cap\boz}\lf|u(y)-u_{B(x, r)\cap\boz}\r|dy\leq Cr\lf(\fint_{B(x, \lambda r)\cap\boz}|\nabla u(y)|^ndy\r)^{\frac{1}{n}}.\]
It shows $\boz$ is a $n$-Poincar\'e domain.

\noindent\underline{\it Case ($n<p<\fz$):} For every $u\in W^{1,p}(\boz)$, $x\in\boz$ and $0<r<\diam(\boz)$, we have 
\begin{eqnarray}
\lf|u(x_1)-u(x_2)\r|&\leq&C|x_1-x_2|^{1-\frac{n}{p}}\lf(\int_{B(x, \lambda r)\cap\boz}|\nabla u(y)^pdy|\r)^{\frac{1}{p}}\nonumber\\
                             &\leq&Cr^{1-\frac{n}{p}}\lf(\int_{B(x, \lambda r)\cap\boz}|\nabla u(y)|^pdy\r)^{\frac{1}{p}}\nonumber
\end{eqnarray}
for almost every $x_1, x_2\in B(x, r)\cap\boz$. Do the integration twice over the set $B(x, r)\cap\boz$. Then the triangle inequality and the fact $\boz$ is intrinsic Ahlfors $n$-regular give the following inequality that
\begin{eqnarray}
\fint_{B(x, r)\cap\boz}\lf|u(y)-u_{B(x, r)\cap\boz}\r|dy&\leq&\fint_{B(x, r)\cap\boz}\fint_{B(x, r)\cap\boz}|u(x_1)-u(x_2)|dx_1dx_2\nonumber\\
                                                                                  &\leq&Cr\lf(\fint_{B(x, \lambda r)\cap\boz}|\nabla u(y)|^pdy\r)^{\frac{1}{p}}.\nonumber
\end{eqnarray} 
It shows $\boz$ is a $p$-Poincar\'e domain. 
\end{proof}

Implied by Lemma \ref{le:unisli}, to prove the sufficient part of Theorem \ref{t1.1}, we should show the domain $\boz$ is quasiconvex and $LLC$. For a metric measure space $(X, d, \mu)$, we say $\mu$ is doubling, if there exists a large enough constant $C>1$ such that for every $x\in X$ and $0<r<\diam(X)$, we have 
\[\mu(B(x, 2r))\leq C\mu(B(x, r)).\] 
It was proved by David and Semmes in \cite{DS:1993} that a complete $p$-Poincar\'e space $(X, d, \mu)$ with $1\leq p<\fz$ and a doubling $\mu$ is quasiconvex. However, a domain $\boz\subset\rn$ obviously is not complete. However, the Sobolev extension property of relative domains can help us to overcome this gap. 
\begin{defn}\label{de:extension}
A domain $\boz\subset\rn$ is called a $W^{1, p}$-extension domain for $1\leq p\leq \fz$, if for every $u\in W^{1, p}(\boz)$, there exists a function $E(u)\in W^{1, p}(\rn)$ such that the restriction of $E(u)$ onto the domain $\boz$ always equals to $u$ and the inequality
\[\|E(u)\|_{W^{1, p}(\rn)}\leq C\|u\|_{W^{1, p}(\boz)}\]
holds with a constant $C$ independent of $u$. 
\end{defn}
The study of geometrical and analytical properties of Sobolev extension domains has been one of the central topics in analysis for long time. In \cite{C:1961, s70}, Calder\'on and Stein proved every Lipschitz domain is a $W^{1, p}$-extension domain for every $1\leq p\leq\fz$. In few years, Jones generalized this result to the class of the so-called $(\epsilon, \delta)$-domains. It is a much larger class than the class of Lipschitz domains. By the result in \cite{hkt08, hkt08b} from Haj\l{}asz, Koskela and Tuominen, for every $W^{1, p}$-extension domain with $1<p<\fz$, there always exists a bounded linear extension operator. Also refer to \cite{S:2006, S:2007} for this result. There are also so many other interesting articles about Sobolev extension theory that we cannot list all of them in references. For those who are interested in it, please search them online. Since the Euclidean space $\rn$ is obviously a $p$-PI space for every $1<p<\fz$, the following lemma is a special case of \cite[Proposition 1.9]{GIZ:2023}. Please refer to  \cite{GIZ:2023} for the definition of $p$-PI spaces.
\begin{lem}\label{le:exten}
Let $\boz\subset\rn$ be an Ahlfors $n$-regular $p$-Poincar\'e domain for $1<p<\fz$. Then it is also a $W^{1, p}$-extension domain. Especially, a local Sobolev-Poincar\'e embedding domain of order $p$ for $1<p<\fz$ is a $W^{1, p}$-extension domain.
\end{lem}
The H\"older inequality directly implies that a $p$-Poincar\'e domain is also a $\tilde p$-Poincar\'e domain for every $1\leq p\leq\tilde p<\fz$. Hence, an Ahlfors $n$-regular $p$-Poincar\'e domain is always a $W^{1, \tilde p}$-extension domain for arbitrary $1\leq p<\tilde p<\fz$. 
Let us recall the definition of $p$-capacity which can be found in \cite{GR:1990, HKM:1993,KUZ:2022,M:book11}.
\begin{defn}\label{de:capacity}
A condenser in a domain $\boz\subset\rn$ is a pair $(E, F)$ of bounded compact connected subsets of $\overline\boz$ with 
$$\dist(E, F)>0.$$ 
Fix $1\leq p<\fz$.The set of $p$-admissible functions for the triple $(E, F;\boz)$ is 
\[\mathcal W_p(E, F; \boz):=\lf\{u\in W^{1, p}(\boz)\cap C(\boz\cup E\cup F): u\big|_E\geq 1, u\big|_F\leq 0\r\}.\]
We define the $p$-capacity of the pair $(E, F)$ with respect to $\boz$ by setting 
\[Cap_p(E, F; \boz):=\inf_{u\in\mathcal W_p(E, F; \boz)}\int_\boz|\nabla u(x)|^pdx.\]
For a pair $(U, V)$ of arbitrary bounded connected open subsets of $\boz$ with $\dist(U, V)>0$, we define the $p$-capacity by setting 
\[Cap_p(U, V; \boz):=\sup_{E\subset U, F\subset F}Cap_p(E, F; \boz),\]
where the supremum is taken over all compact connected subsets of $\boz$ with $E\subset U$ and $F\subset V$.
\end{defn}
For a $W^{1, p}$-extension domain, the following lemma tells us the $p$-capacity with respect to the full space $\rn$ can be controlled from above by the $p$-capacity with respect to the domain. The readers can find a proof in \cite{K:1990}
\begin{lem}\label{le:capin}
Let $\boz\subset\rn$ be a bounded $W^{1, p}$-extension domain for $1<p<\fz$. Then for every bounded connected open subsets $U, V\subset\boz$ with 
\[\dist(U, V)>0,\] 
the inequality 
\begin{equation}\label{eq:capin}
Cap_p\lf(U, V; \rn\r)\leq C Cap_p\lf(U, V; \boz\r)
\end{equation}
holds with a constant $C$ independent of $U$ and $V$.
\end{lem}
In the existing literature, the domain supporting (\ref{eq:capin}) is called a $p$-quasiextremal distance domain. See \cite{gm85b} for $p=n$ and \cite{K:1990} for arbitrary $1<p<\fz$. It was also proved in \cite{K:1990} that a $p$-quasiextremal distance domain for $n\leq p<\fz$ must be quasiconvex. Now, we are ready to prove that a local Sobolev-Poincar\'e imbedding domain of order $p$ for $1\leq p<\fz$ must be quasiconvex.
\begin{lem}\label{le:quasiconvex}
Let $\boz\subset\rn$ be a bounded local Sobolev-Poincar\'e imbedding domain of order $p$ for $1\leq p<\fz$. Then it must be quasiconvex. 
\end{lem}
\begin{proof}
Let $\boz\subset\rn$ be a local Sobolev-Poincar\'e imbedding domain of order $p$ for $1\leq p<\fz$. The Lemma \ref{l2.3} tells us that it is Ahlfors $n$-regular. By Lemma \ref{le:SPtoP} and the H\"older inequality, $\boz$ is a $\tilde p$-Poincar\'e domain for some $n\leq\tilde p<\fz$. Then, by Lemma \ref{le:exten}, it is a $W^{1,\tilde p}$-extension domain. Finally, by Lemma \ref{le:capin}, it satisfies the inequality (\ref{eq:capin}) for $\tilde p$-capacity and is a $\tilde p$-quasiextremal distance domain. By the result in \cite{K:1990} from Koskela, $\boz$ is quasiconvex. 
\end{proof}

\section{Proof of Theorem \ref{t1.1}}
\subsection{From uniform condition to local Sobolev-Poincar\'e imbedding}
In this section, we will prove the sufficient part of Theorem \ref{t1.1}. We will show that a uniform domain is also a local Sobolev-Poincar\'e imbedding domain of order $p$ for all $1\leq p\leq n$. 

\begin{thm}
Let $\boz\subset\rn$ be a uniform domain. Then it is also a local Sobolev-Poincar\'e imbedding domain of order $p$ for every $1\leq p\leq n$.
\end{thm}
\begin{proof} 
Let $\Omega$ be a $\epsilon_0$-uniform domain for $0<\epsilon_0\leq 1$. Fix $x_0\in\boz$ and $0<r<\diam(\boz)$. By Lemma \ref{l2.9}, there always exists a $c_0$-John domain $\Omega_r \subset \Omega$ with $c_0$ only depends on $\epsilon_0$ and a constant $\lambda>1$ such that
$$B(x_0,r) \cap \Omega \subset \Omega_r \subset B(x_0, \lambda r) \cap \Omega.$$ 
Depending on if $p$ equals to $n$ or not, we divide the argument below into two cases.

\noindent\underline{\it Case ($ 1\leq p < n$):}
Let $u\in{W}^{1,p}(\Omega)$ be arbitrary. Since $B(x_0, r) \cap \Omega \subset \Omega_r$, we have
\begin{align*}
 \inf \limits _{c \in\mathbb{R}} \left( \int_{B(x_0, r) \cap \Omega} |u(x)- c |^{\frac{np}{n-p}}\, dx \right)^{\frac{n-p}{np}}
&\leq \inf \limits _{c \in\mathbb{R}} \left(\int_{\Omega_r}|u(x)-c|^{\frac{np}{n-p}}\,dx\right)^{\frac{n-p}{np}}.
\end{align*}
Since $\boz_r$ is a $c_0$-John domain with some constant $c_0$ only depends on $\epsilon_0$, by the result in \cite{b89} by Bojarski, there exists $C \geq 1$ independent of $u, x_0$ and $r$ such that
\begin{align*}
\inf \limits _{c \in\mathbb{R}} \left(\int_{\Omega_r}|u(x)-c|^{\frac{np}{n-p}}\,dx\right)^{\frac{n-p}{np}}
\leq C \left(\int_{\Omega_r} |\nabla u(x)|^{p}\,dx\right)^{\frac{1}{p}}
\leq C \left( \int_{B(x_0, \lambda r) \cap \Omega} | \nabla u(x) |^{p} \,dx\right)^{\frac{1}{p}}.
\end{align*}
Hence we have
\begin{align*}
\inf \limits _{c \in\mathbb{R}} \left( \int_{B(x_0, r) \cap \Omega} |u(x)- c |^{\frac{np}{n-p}}\, dx \right)^{\frac{n-p}{np}}
 \leq C \left( \int_{B(x_0, \lambda r) \cap \Omega} | \nabla u(x) |^{p} \,dx\right)^{\frac{1}{p}}.
\end{align*}
It implies $\boz$ is a local Sobolev-Poincar\'e imbedding domain of order $p$ for every $1\leq p<n$.

\noindent\underline{\it Case ($p=n$):} Let $u\in{W}^{1,n}(\Omega)$ be arbitrary. Since
$$B(x_0, r) \cap \Omega \subset \Omega_r\subset B(x_0, \lambda r)\cap\boz,$$
we have
\begin{align*}
\inf\limits_{c\in\mathbb R} \int_{B(x_0,r)\cap\Omega} \exp\left(\frac{A|u(x)-c|}{\|\nabla u\|_{L^{n}(B(x, \lambda r)\cap\Omega)}}\right)^{\frac{n}{n-1}}\,dx
&\leq\inf\limits_{c\in\mathbb R} \int_{\Omega_r} \exp\left(\frac{A|u(x)-c|}{\|\nabla u\|_{L^{n}(\Omega_{r})}}\right)^{\frac{n}{n-1}}\,dx.
\end{align*}
By \cite[Lemma 3.11]{gm85a}, we know $\Omega_r $ is also a $0$-carrot domain. By the result dues to Smith and Stegenga \cite{ss91}, we have
\begin{align*}
\inf \limits_{c\in\mathbb{R}} \int_{\Omega_r} \exp\left(\frac{A|u(x)-c|}{\|\nabla u\|_{L^{n}(\Omega_{r})}}\right)^{\frac{n}{n-1}}\,dx  \leq C|\Omega_r|
\end{align*}
for a constant $C$ independent of $x_0, r ,u$. Hence, we have
\begin{align*}
\inf \limits_{c\in\mathbb{R}} \int_{B(x_0,r)\cap\Omega} \exp \left(\frac{A|u(x)-c|}{\|\nabla u\|_{L^{n}(B(x_0, \lambda r)\cap\Omega)}}\right)^{\frac{n}{n-1}}\,dx
&\leq C |\Omega_r| \leq C |B(x_0, \lambda r)\cap\Omega|\\
& \leq C |B(x_0, \lambda r)| \leq C |B(x_0, r)|.
\end{align*}
It implies the uniform domain $\boz$ is also a local Sobolev-Poincar\'e imbedding domain of order $n$.
\end{proof}

\subsection{From local Sobolev-Poincar\'e imbedding to $(LLC)$}
In this section, we prove the necessary part of Theorem \ref{t1.1}. We will show that under the assumption that the domain satisfies the slice condition, a local Sobolev-Poincar\'e imbedding domain  of order $p$ for $1\leq p\leq n$ is also a uniform domain. By Lemma \ref{le:unisli}, it suffices to prove that it is quasiconvex and $LLC$. The Lemma \ref{le:quasiconvex} already shows it is quasiconvex. By the definition, the fact that a domain is quasiconvex implies that this domain is also $LLC(1)$. Hence, we only need to show it is $LLC(2)$. According to $p$ equals to $n$ or not, we divide the argument into two separated cases.

First, let us consider the case that $1\leq p<n$.
\begin{thm}\label{t3.3}
 Let $u\in W ^{1, p}(\Omega )$ be arbitrary for $1\leq p<n$. If there exist constants $C, \lambda>1$ such that the inequality
\begin{eqnarray}\label{e3.2}
\inf \limits _{c\in\mathbb R} \left( \int_{B(x_0, r)\cap \Omega} |u(x)- c |^{\frac{np}{n-p}}\, dx \right)^{\frac{n-p}{np}}
 \leq C \left( \int_{B(x_0, \lambda r)\cap \Omega} | \nabla u (x)|^{p}\,dx \right)^{\frac{1}{p}}
\end{eqnarray}
holds for every $x_0 \in \Omega$ and $0<r < \fz$, then $\Omega $ is $LLC$.
\end{thm}

\begin{proof}

As we know, the Lemma \ref{le:quasiconvex} tells us that $\boz$ is quasiconvex. By the definitions, a quasiconvex domain must be $LLC(1)$. So we only need to show that $\boz$ is $LLC(2)$. Without loss of generality, we can only consider the case that $x_1 , x_2 \in \Omega \backslash B (x_0, r)$ for some  $x_0 \in \Omega $ and $0<r<\diam(\boz)$. Furthermore, we can assume that $x_1, x_2 \in \partial B (x_0, r) \cap \Omega $. Assume that $x_1$ and $x_2$ are not in the same component of $\Omega \backslash{B(x_0, br )}$ for $b \in (0,1)$, it suffices to find a uniform lower bound for $b$. We may assume $b<\frac{1}{16}$. Denote by $\Omega_i$ the connected component of $\Omega \backslash{B(x_0, br)}$ which contains $x_i$ for $i=1, 2$, then $\Omega_1 \cap \Omega_2 = \emptyset$. Without loss of generality, we can only consider the case that $0<r\leq r_o$ where $r_o$ comes from the fact that $\boz$ is intrinsic Ahlfors $n$-regular. If we proved that for every $0<r\leq r_o$, there exists $0<a<1$ such that $x_1$ and $x_2$ can be connected by a locally rectifiable curve $\gamma\subset\boz\setminus B(x_0, ar)$. Then for every $r_o<r<\diam(\boz)$, $x_1, x_2$ can be connected by a locally rectifiable curve
$$\gamma\subset\boz\setminus B\lf(x_0, \frac{ar_o}{\diam(\boz)}r\r).$$

Fix $0<r\leq r_o$. Define $F_i$ to be the connected component of $\Omega \cap B(x_i , \frac{r}{2})$ which contains $x_i$ for $i=1, 2$.
Obviously,  we have $F_1, F_2 \subset B (x_0, 2r)$, $F_1 \cap F_2= \emptyset$ and $F_i \cap B(x_0, \frac{r}{2})=\emptyset $ for $i=1,2$. Since $\boz$ is intrinsic Ahlfors $n$-regular, there exists a sufficient small constant $0<c<1$ such that we have
\[|F_1|\geq cr^n\ {\rm and}\ |F_2|\geq cr^n.\]
Define a function $u$ on $\boz$ by setting
\begin{align*}
 u(x)=\left\{
\begin{array}{ll}
 1  ,      &    x \in \Omega_2 \backslash B(x_0, \frac{r}{2}) \\
\frac{1}{\log \frac{1}{2b}}log \frac{|x- x_0|}{br} , \,\,     &    x \in \Omega_2  \cap \left(B(x_0, \frac{r}{2})\backslash B(x_0, br )\right)\\
 0,       &      elsewhere.
\end{array} \right.
\end{align*}
Obviously, this function is well defined. Noting that
\begin{align*}
 |\nabla u (x) | =\frac{1}{\log \frac{1}{2b}}\frac{\chi _{\Omega_2  \cap \left(B(x_0, \frac{r}{2})\backslash B(x_0, br )\right)}(x)}{ |x- x_0 |},
 \end{align*}
we have
\begin{eqnarray}\label{eq:upper1}
\left( \int_{B(x_0, \lambda r) \cap \Omega} | \nabla u(x)|^{p}dx \right)^{\frac{1}{p}}
 && \leq \frac{1}{\log \frac{1}{2b}} \left( \int_{ \Omega_2  \cap \left(B(x_0, \frac{r}{2})\backslash B(x_0, br )\right)} | x- x_0 |^{-p}dx \right)^{\frac{1}{p}}\\
 && \leq \frac{1}{\log \frac{1}{2b}}  r^{\frac{n-p}{p}} \left[\left(\frac{1}{2}\right)^{n-p} - b^{n-p} \right]^{\frac{1}{p}}\nonumber\\
 && \leq \frac{Cr^{\frac{n-p}{p}}}{\log\frac{1}{2b}}.\nonumber
\end{eqnarray}
For each $c \in \rr$, $|u- c| \geq \frac{1}{2}$ holds either on $\Omega_2\cap\lf(B(x_0, r)\setminus B\lf(x_0, \frac{r}{2}\r)\r)$ or on $\Omega_1\cap\lf(B(x_0, r)\setminus B\lf(x_0, \frac{r}{2}\r)\r)$. Since $\boz$ is intrinsic Ahlfors $n$-regular, we have both
\[\lf|\boz_i\cap\lf(B(x_0, r)\setminus B\lf(x_0, \frac{r}{2}\r)\r)\r|\geq cr^n\]
for $i=1, 2$ and a sufficient small constant $0<c<1$. Hence, we have
\begin{eqnarray}\label{eq:upper2}
\inf \limits _{c\in\mathbb R} \left( \int_{B(x_0, r) \cap \Omega} |u(x)-c|^{\frac{np}{n-p}}\, dx \right)^{\frac{n-p}{np}}\geq cr^{\frac{n-p}{n}},
\end{eqnarray}
for a sufficient small $0<c<1$. By combining inequalities (\ref{e1.w1}), (\ref{eq:upper1}) and (\ref{eq:upper2}), we obtain
\[\frac{1}{\log\frac{1}{2b}}>c\]
for some sufficient small $0<c<1$. It gives a lower bound to $b$ and proves $\boz$ is $LLC(2)$.
\end{proof}

Next, we consider the case that $p=n$.
\begin{thm}\label{t3.4}
 If for every $u\in W ^{1, n}(\Omega )$, there exist constants $C, \lambda>1$  the following inequality
\begin{eqnarray}\label{e3.3}
\inf \limits_{c \in\mathbb{R}} \int_{B(x_0,r)\cap\Omega} \exp\left(\frac{A|u(x)-c|}{\|\nabla u\|_{L^{n}(B(x_0, \lambda r)\cap\Omega)}}\right)^{\frac{n}{n-1}}\,dx
\leq C |B(x_0,r)|
\end{eqnarray}
holds with every $x_0 \in \Omega$ and $0<r < \diam \Omega $, then $\Omega $ is $LLC$.
\end{thm}

\begin{proof}

As we know,  the Lemma \ref{le:quasiconvex} implies that $\boz$ is quasiconvex. By the definitions, a quasiconvex domain must be $LLC(1)$. So we only need to show that $\boz$ is $LLC(2)$ below. All argument is almost as same as the argument to show that $\boz$ is $LLC(2)$ in the proof of Theorem \ref{t3.3}. The difference is that we need to replace the inequality (\ref{e1.w1}) by the inequality (\ref{e1.w2}) and to estimate the $n$-Dirichlet energy of the test function $u$. Some simple computation gives
\begin{eqnarray*}
\left( \int_{B(x_0, \lambda r) \cap \Omega} | \nabla u(x)|^{n}dx \right)^{\frac{1}{n}}
 && \leq \frac{C}{\log \frac{1}{2b} }\left( \int_{ \Omega_2  \cap \left(B(x_0, \frac{r}{2})\backslash B(x_0, br )\right)} | x- x_0 |^{n}dx \right)^{\frac{1}{n}}\\
 && \leq C \left(\log \frac{1}{2b} \right)^{\frac{1-n}{n}}
\end{eqnarray*}
and
\begin{eqnarray*}
\int_{B(x_0, r) \cap \Omega} \exp\left(\frac{A|u(x)-c|}{\|\nabla u\|_{L^{n}(B(x_0, \lambda r)\cap \tilde{\Omega}_{1})}}\right)^{\frac{n}{n-1}}\,dx
\geq c\exp\lf(C\log\frac{1}{2b}\r)r^n
\end{eqnarray*}
for some sufficient small $0<c<1$. Then the inequality (\ref{e3.3}) gives
\[1>c\exp\lf(C\log\frac{1}{2b}\r).\]
It implies a uniform lower bound to $b$ and $\boz$ is $LLC(2)$.
\end{proof}

\section{Proof of Theorem \ref{th:2}}
In this section, we prove Theorem \ref{th:2}.
\begin{proof}[Proof of Theorem \ref{th:2}]
First ,we show a local Sobolev-Poincar\'e imbedding domain of order $p$ for $n<p<\fz$ with slice condition is also an $\alpha$-cigar domain for $\alpha=(p-n)/(p-1)$. For arbitrary $x_1, x_2\in\boz$, we can find $x\in\boz$ and $0,r<\diam(\boz)$ with $x_1, x_2\in B(x, r)\cap\boz$. Then, we have
\begin{eqnarray*}
|u(x_1)-u(x_2)|&\leq& C|x_1-x_2|^{1-\frac{n}{p}}\lf(\int_{B(x, \lambda r)\cap\boz}|\nabla u(y)|^pdy\r)^{\frac{1}{p}}\\
&\leq& C|x_1-x_2|^{1-\frac{n}{p}}\lf(\int_{\boz}|\nabla u(y)|^pdy\r)^{\frac{1}{p}}.
\end{eqnarray*}
So $\boz$ is also a global Sobolev-Poincar\'e imbedding domain of order $p$. By the result in \cite{bsk96}, $\boz$ is an $\alpha$-cigar domain for $\alpha=(p-n)/(p-1)$.

Next, we show an $\alpha$-cigar domain $\boz$ with $\alpha=(p-n)/(p-1)$ is a local Sobolev-Poincar\'e imbedding domain of order $p$ for $n<p<\fz$. By \cite[Lemma 2.2]{bsk96}, for every pair of points $x_1, x_2\in\boz$, there exists curve $\gamma\subset\boz$ joining $x_1$ and $x_2$ with
\[l(\gamma)\leq C|x_1-x_2|\]
for some uniform constant $C>1$, and
\[\int_{\gamma}\dist(z, \partial\boz)^{\alpha-1}|dz|\leq C|x_1-x_2|^\alpha.\]
It implies that there exists a large enough $C>1$ such that for every $x\in\boz$ and $0<r<\diam(\boz)$, $B(x, Cr)\cap\boz$ contains a connected component, which contains $B(x, r)\cap\boz$, is an $\alpha$-cigar domain with $\alpha=(p-n)/(p-1)$. Hence, for every $x_1, x_2\in B(x, r)\cap\boz$, we have
\[|u(x_1)-u(x_2)|\leq C|x_1-x_2|^{1-\frac{n}{p}}\lf(\int_{B(x, \lambda r)\cap\boz}|\nabla u(y)|^pdy\r)^{\frac{1}{p}}.\]
It implies $\boz$ is a local Sobolev-Poincar\'e imbedding domain of order $p$.
 \end{proof}


\medskip


\end{document}